\documentclass[12pt,a4paper,english]{amsart}
\usepackage{amsmath}
\usepackage{amsthm}
\usepackage{amsfonts}
\usepackage{amscd}
\usepackage{amssymb}
\usepackage{array}
\usepackage[mathscr]{eucal}
\usepackage{hyperref}
\usepackage{caption}

\numberwithin{equation}{section}
\usepackage[margin=2.9cm]{geometry}

\theoremstyle{plain}
\newtheorem{teo}{Theorem}[section]
\newtheorem{lema}[teo]{Lemma}
\newtheorem{coro}[teo]{Corollary}
\newtheorem{prop}[teo]{Proposition}

 \theoremstyle{definition}
\newtheorem{defi}[teo]{Definition}
\newtheorem{conj}[teo]{Conjecture}
\newtheorem{obs}[teo]{Remark}
\newtheorem{ej}[teo]{Example}

\newcommand{\Z}{\mathbb{Z}}

\allowdisplaybreaks

\begin{document}

\title{Hypersimplices are Ehrhart Positive}

\author[L. Ferroni]{Luis Ferroni}
\thanks{The author is supported by the Marie Sk{\l}odowska-Curie PhD fellowship as part of the program INdAM-DP-COFUND-2015.}

\address{Universit\`a di Bologna, Dipartimento di Matematica, Piazza di Porta San Donato, 5, 40126 Bologna BO - Italia} 

\email{ferroniluis@gmail.com\\ luis.ferronirivetti2@unibo.it}

\subjclass[2010]{05B35, 52B20, 11B73}

\begin{abstract} 
	We consider the Ehrhart polynomial of hypersimplices. It is proved that these polynomials have positive coefficients and we give a combinatorial formula for each of them. This settles a problem posed by Stanley and also proves that uniform matroids are Ehrhart positive, an important and yet unsolved particular case of a conjecture posed by De Loera et al. To this end, we introduce a new family of numbers that we call \textit{weighted Lah numbers} and study some of their properties.\\
	
	\smallskip
    \noindent {\scshape Keywords.} Hypersimplices, Ehrhart polynomials, Polytopes, Uniform Matroids.
\end{abstract}

\maketitle

\section{Introduction} 

Let us fix two positive integers $n$ and $k$ with $k\leq n$. The $(k,n)$-hypersimplex, denoted by $\Delta_{k,n}$ is defined by:
    \[ \Delta_{k,n} := \left\{x\in [0,1]^n : \sum_{i=1}^n x_i = k\right\}.\]

This polytope appears naturally in several contexts within geometric and algebraic combinatorics. For example, it can be seen as a weight polytope of the fundamental representation of the general linear groups $\operatorname{GL}(n)$ or as the basis polytope of the uniform matroid $U_{k,n}$.\\

The basis polytope (also known as the \textit{matroid polytope}) of a matroid is defined as the convex hull of the indicator functions of its bases \cite{GGMS}, \cite{welsh}. It encodes all the information about the matroid, hence providing a geometric point of view on matroidal notions and problems. These polytopes are relevant objects in algebraic combinatorics, toric geometry, combinatorial optimization, and as of today a matter of extensive research \cite{ardila}, \cite{fink}, \cite{feichtner}, \cite{deloera}, \cite{knauer2}.
The uniform matroid $U_{k,n}$ is in particular defined as the matroid on the set $\{1,2,\ldots,n\}$ having set of bases $\mathscr{B}:= \{ B\subseteq \{1,\ldots,n\} : |B|=k\}$ and, as we said above, its basis polytope coincides with the hypersimplex $\Delta_{k,n}$.\\

An important invariant of a polytope $\mathscr{P}$ whose vertices lie in $\mathbb{Z}^d$, is the so-called \textit{Ehrhart polynomial} \cite{ehrhart}, \cite{beck}. It is defined as the polynomial $p\in \mathbb{Q}[t]$ such that $p(t) = |\mathbb{Z}^d\cap t\mathscr{P}|$ for $t\in\mathbb{Z}_{\geq 0}$, being $t\mathscr{P}$ the \textit{dilation} of $\mathscr{P}$ with respect to the origin by the factor $t$. In particular, since the basis polytope of a matroid has vertices with $0/1$ coordinates, we can consider its Ehrhart polynomial, which we will sometimes refer as the Ehrhart Polynomial of the matroid itself.\\

In the paper \cite[Conjecture 2(B)]{deloera}, the authors posed the following conjecture:

\begin{conj}[De Loera et al]\label{conje}
	Let $\mathscr{P}(M)$ be the basis polytope of a matroid $M$. Then the coefficients of the Ehrhart polynomial of $\mathscr{P}(M)$ are positive.
\end{conj}

Also in \cite[Lemma 29]{deloera} the authors proved that hypersimplices $\Delta_{2,n}$ have an Ehrhart polynomial with positive coefficients. Their proof is based on inequalities, using a result of Katzman \cite{katzman} regarding a formula for the Ehrhart polynomial of hypersimplices, found in the context of algebras of Veronese type. Also, as a corollary of a result in \cite{ohsugi} regarding the complex roots of the Ehrhart polynomial of $\Delta_{3,n}$ one can also conclude the positivity of the Ehrhart coefficients for this particular case.\\

We will restate Katzman's formula in the following section, prove it using generating functions and exploit it to prove the main result of this article.

\begin{teo}\label{main}
	The coefficients of the Ehrhart polynomial of all hypersimplices $\Delta_{k,n}$ are positive.
\end{teo}

Moreover, we are going to introduce the notion of \textit{weighted Lah number} and to provide a combinatorial formula for each coefficient in terms of them. We thus settle an open problem posed in Richard Stanley's book \cite[Ch. 4, Problem 62e]{stanley}. 

It is worth noting that according to \cite{stanleyeulerian} the calculation of the principal coefficient of these polynomials (it is, the normalized volume of the hypersimplex) dates back to Laplace, though apparently he did not do it explicitly. The principal coefficients are what in the literature are called Eulerian numbers and several combinatorial interpretations exist for them \cite{stanley,knuth}. Thanks to some properties of the Ehrhart polynomials and the fact that the facets of a hypersimplex are also hypersimplices, there was also a combinatorial interpretation of the second highest coefficient in terms of Eulerian numbers. However the rest of them remained elusive. More recently, independent proofs of the positivity of the linear coefficient of the Ehrhart polynomial of all hypersimplices were found in \cite{liutodd} and \cite{jochemko2019generalized}.\\

In the paper \cite{fuliu} the Conjecture \ref{conje} has been strengthened and reformulated in a more general setting, asserting that indeed all \textit{generalized permutohedra} \cite{postnikov} are Ehrhart positive. Since it is known \cite{ardila} that this family contains all matroid polytopes, this conjecture is indeed stronger. Also, in \cite{liu} there is a survey on the families of polytopes that are known to be Ehrhart positive, and those that are conjectured to also have this property. \\

The Ehrhart $h^*$-polynomial \cite{beck,stanleyhstar} of hypersimplices is itself a rich object of study. For example, the unimodality of its coefficients is still an open problem \cite{braununimodality}. Recently the author conjectured that the $h^*$-polynomial of all matroid polytopes (in particular of hypersimplices) are real-rooted \cite{ferroni2}. Some combinatorial interpretations for the coefficients of the $h^*$-polynomials of hypersimplices were found in \cite{nanli}, \cite{early2017conjectures} and \cite{kim}.

\section{The Ehrhart Polynomial of Hypersimplices}

We will base our computations on a formula found by Katzman in \cite[Corollary 2.2]{katzman} for $E_{k,n}$. We will provide a generating-function based proof. 

\begin{teo}\label{katz}
	The Ehrhart Polynomial $E_{k,n}(t)$ of the hypersimplex $\Delta_{k,n}$ is given by:
		\begin{equation}\label{formula} 
			E_{k,n}(t) = \sum_{j=0}^{k-1} (-1)^j \binom{n}{j} \binom{(k-j)t+n-1-j}{n-1}.
		\end{equation}
\end{teo}

It is not at all apparent from this formula that the coefficients of the polynomial $E_{k,n}$ are positive. Indeed the alternating factor $(-1)^j$ and the fact that the variable $t$ appears inside a binomial coefficient which in turn for $j>1$ is a polynomial with some negative coefficients do not permit us to see this fact directly. Before proving Theorem \ref{katz} we establish a useful Lemma.

\begin{lema}\label{formulita}
	If $1\leq k\leq n-1$ and $t\geq 0$, then the coefficient of $x^{kt}$ in the polynomial $(1+x+x^2+\ldots+x^t)^n$ is exactly $E_{k,n}(t)$.
\end{lema}

\begin{proof}
    By definition, the polynomial $E_{k,n}(t)$ counts the number of elements in the set $t\Delta_{k,n} \cap \mathbb{Z}^n $. This set can be rewritten as:
        \[\left \{ y \in \{0,1,\ldots,t\}^n : \sum_{i=1}^n y_i = kt\right\}.\]
    But notice that the coefficient of $x^{kt}$ in the product
        \[(1+x+x^2+\ldots+x^t)^n= \underbrace{(1+x+x^2+\ldots+x^t) \cdot \ldots \cdot (1+x+x^2+\ldots+x^t)}_{n \text{ times}} \]
    is exactly the number of ways of choosing a sequence of $n$ elements in the set $\{0,1,\ldots,t\}$ in such a way that their sum is exactly $kt$. That is exactly the cardinality of our set.
\end{proof}

Recall that if one has a (formal) power series $f(x):=\sum_{j=0}^{\infty} a_j x^j$, it is customary to use the notation $[x^\ell]f(x):=a_{\ell}$.

\begin{proof}[Proof of Theorem \ref{katz}]
	We will use generating functions to compute the coefficient of $x^{kt}$ in $(1+x+\ldots+x^t)^n$ and then we will use the preceding Lemma. Notice that:
		\begin{align*} 
			[x^{kt}] \left(1+x+\ldots+x^t\right)^n &= [x^{kt}]\left(\frac{1-x^{t+1}}{1-x}\right)^n\\
			&= [x^{kt}] \left((1-x^{t+1})^n \cdot \frac{1}{(1-x)^n}\right)
		\end{align*}
	So writing $(1-x^{t+1})^n = \sum_{j=0}^n (-1)^j \binom{n}{j} x^{(t+1)j}$ and $\frac{1}{(1-x)^n} = \sum_{j=0}^{\infty} \binom{n-1+j}{n-1} x^j$, the coefficient of $x^{kt}$ in this product can be computed as a convolution:
		\[ \sum_{j=0}^{k-1} (-1)^j \binom{n}{j} \binom{n-1+(k-j)t-j}{n-1},\]
		where the sum ends in $k-1$ since in the first of our two formal series we have $x^{(t+1)j}$ and we are computing the coefficient of $x^{kt}$. Also, the second binomial coefficient in our expression comes from the fact that $(t+1)j + ((k-j)t-j)=kt$. 
\end{proof}

\section{Weighted Lah Numbers}

In this section we develop some useful tools to prove Theorem \ref{main}. We recall the definition of \textit{Lah numbers} (also known as \textit{Stirling Numbers of the 3rd kind}). 

\begin{defi}
	The \textit{Lah number} $L(n,m)$ is defined as the number of ways of partitioning the set $\{1,2,\ldots,n\}$ in exactly $m$ linearly ordered blocks. We will denote the set of all such partitions by $\mathscr{L}(n,m)$.
\end{defi}

\begin{ej}
	$L(3,2)=6$ because we have the following possible partitions:
		\[ \{(1,2),(3)\}, \{(2,1),(3)\},\]
		\[ \{(1,3),(2)\}, \{(3,1),(2)\},\]
		\[ \{(2,3),(1)\}, \{(3,2),(1)\}.\]
\end{ej}

If $\pi$ is a partition of $\{1,\ldots,n\}$ in $m$ linearly ordered blocks, for any of these blocks $b$, we will write $b\in \pi$. So, for example $(2,3)\in \{(2,3),(1)\}$. Also, we will use the notation $|b|$ to denote the number of elements in $b$.

\begin{obs}
	We have the equality $L(n,m)=\frac{n!}{m!}\binom{n-1}{m-1}$. This can be proven easily by a combinatorial argument as follows. Order the $n$ numbers on the set in any fashion. To get the partition we can put $m-1$ divisions in any of the $n-1$ spaces between two consecutive numbers. Then divide by $m!$, the number of ways of ordering all the blocks.
\end{obs}

There already exist some generalizations of these numbers \cite{genlah}. We will introduce a new one that we will call \textit{weighted Lah numbers}.

\begin{defi}
	Let $\pi$ be a partition of the set $\{1,\ldots,n\}$ into $m$ linearly ordered blocks. We define the \textit{weight of $\pi$} by the following formula:
		\[ w(\pi) := \sum_{b\in\pi} w(b),\]
	where $w(b)$ is the number of elements in $b$ that are smaller (as positive integers) than the first element in $b$.
\end{defi}

\begin{ej}\label{ejemplito}
	Among the $6$ partitions that we have seen that exist of $\{1,2,3\}$ into $2$ blocks, we have:
		\[ w(\{(1,2),(3)\}) = 0 + 0 = 0,\;\; w(\{(2,1),(3)\}) = 1 + 0 = 1,\]
		\[ w(\{(1,3),(2)\}) = 0 + 0 = 0,\;\;  w(\{(3,1),(2)\}) = 1 + 0 = 1,\]
		\[ w(\{(2,3),(1)\}) = 0 + 0 = 0,\;\;  w(\{(3,2),(1)\}) = 1 + 0 = 1.\]
	Note that there are exactly $3$ of these partitions of weight $0$ and exactly $3$ of weight $1$.
\end{ej}

\begin{defi}
	We define the \textit{weighted Lah Numbers} $W(\ell,n,m)$ as the number of partitions of weight $\ell$ of $\{1,\ldots,n\}$ into exactly $m$ linearly ordered blocks. 
\end{defi}

\begin{ej}
	Rephrasing the conclusion of the Example \ref{ejemplito}, we have that $W(0,3,2)=3$ and $W(1,3,2)=3$.
\end{ej}

\begin{table}
	\parbox{.45\linewidth}{\begin{tabular}{>{$}l<{$}|*{5}{c}}
		\multicolumn{1}{l}{$m$} &&&&&\\\cline{1-1} 
		1 &24&24&24&24&24\\
		2 &50&70&70&50\\
		3 &35&50&35&&\\
		4 &10&10&&&\\
		5 &1&&&&\\\hline
		\multicolumn{1}{l}{} &0&1&2&3&4\\\cline{2-6}
		\multicolumn{1}{l}{} &\multicolumn{5}{c}{$\ell$}
	\end{tabular}
	\caption{$W(\ell,5,m)$}}
	\quad\quad\quad
	\parbox{.45\linewidth}{\begin{tabular}{>{$}l<{$}|*{6}{c}}
		\multicolumn{1}{l}{$m$} &&&&&&\\\cline{1-1} 
		1 &120&120&120&120&120&120\\
		2 &274&404&444&404&274\\
		3 &225&375&375&225&&\\
		4 &85&130&85&&&\\
		5 &15&15&&&&\\
		6 &1&&&&&\\\hline
		\multicolumn{1}{l}{} &0&1&2&3&4&5\\\cline{2-7}
		\multicolumn{1}{l}{} &\multicolumn{6}{c}{$\ell$}
	\end{tabular}
	\caption{$W(\ell,6,m)$}}
\end{table}

The set of all partitions of $\{1,\ldots,n\}$ into $m$ linearly ordered blocks and weight $\ell$ will be denoted by $\mathscr{W}(\ell,n,m)$.

\begin{obs}
	Observe that $W(\ell,n,m)\neq 0$ only for $0\leq \ell \leq n-m$. This is because the maximum weight can be obtained by ordering every block in such a way that its maximum element is on the first position. Also, we have the following:
		\[ W(0,n,m) = {n \brack m}\]
	where the brackets denote the \textit{(unsigned) Stirling numbers of the first kind} \cite{knuth}. This can be proven combinatorially by noticing that for every permutation with exactly $m$ cycles, we can present it in a unique way as a partition of $\{1,\ldots,n\}$ into $m$ linearly ordered blocks having every block its minimum element in the first position.
\end{obs}

\begin{obs}
	We have symmetry, namely:
		\[ W(\ell,n,m) = W(n-m-\ell,n,m).\]
	This equality is a consequence of the fact that for $\pi \in \mathscr{W}(\ell,n,m)$ we can associate bijectively an element $\pi'\in \mathscr{W}(n-m-\ell,n,m)$ as follows. In $\pi$ interchange the positions of $1$ and $n$, of $2$ and $n-1$, and so on. What one gets is exactly a partition of weight $n-m-\ell$.
\end{obs}

It is possible to obtain many recurrences to compute $W(\ell,n,m)$ recursively. For instance we include the following:

\begin{prop}
	The following recurrence holds for $n,m\geq 2$:
	\[W(\ell,n,m) = (n-1) W(\ell-1,n-1,m) + \sum_{j=0}^{n-1} \binom{n-1}{j} j! W(\ell,n-1-j,m-1).\]
\end{prop}

\begin{proof}
	Every $\pi\in\mathscr{W}(\ell,n,m)$ has the number $1$ inside a block. If this number is \textit{not} the first element of its block, this means that if we remove it from $\pi$ we end up getting an element of $\mathscr{W}(\ell-1,n-1,m)$ (with every element shifted by one). Analogously, we can pick an element of $\mathscr{W}(\ell-1,n-1,m)$ (which we think of as having every element shifted by one) and reconstruct an element of $\mathscr{W}(\ell,n,m)$ by adjoining the element $1$ in such a way that it is not the first element of a block. There are $n-1$ possibilities of where to put the number $1$ to get an element of $\mathscr{W}(\ell,n,m)$. So we get the first summand.
	
	The remaining cases to consider are those on which $1$ is the first element of its block. In this case we choose $j$ elements to be in this block, and in every possible order of these elements, the block will always have weight $0$. So the remaining $n-j-1$ elements will have to be arranged in $m-1$ blocks of total weight $\ell$.
\end{proof}

\begin{obs}
	The last proposition tells us that if we make the subtraction $W(\ell,n,m)-(n-1)W(\ell,n-1,m)$ we end up getting an expression for which the sum cancels out to give just the recurrence:
		\begin{align*} 
		W(\ell,n,m) &= (n-1)W(\ell-1,n-1,m) + (n-1)W(\ell,n-1,m) \\ 
		&\;+ W(\ell,n-1,m-1) -(n-1)(n-2)W(\ell-1,n-2,m).
		\end{align*}
\end{obs}

We establish now a bivariate generating function for $W(\ell,n,m)$ for a fixed $m$. 

\begin{teo}\label{genfun}
	We have the equality:
	\[W(\ell,n,m) = \frac{n!}{m!}[x^n s^\ell] \left(\tfrac{1}{(1-s)^m} \left(\log\left(\tfrac{1}{1-x}\right) - \log\left(\tfrac{1}{1-sx}\right)\right)^m\right) \]
\end{teo}

\begin{proof}
	Notice that it suffices to prove that:
	\begin{equation} \label{critica}
	W(\ell,n,m) = \frac{n!}{m!} [x^n s^{\ell}] \left( \sum_{k=1}^{\infty} \frac{x^k}{k} (1+s+\ldots+s^{k-1})\right)^m.
	\end{equation}
	This is because using the formula for the geometric series, the sum in the parentheses can be rewritten as $\frac{1}{1-s}\left(\sum_{k=1}^{\infty} \frac{x^k}{k} - \sum_{k=1}^{\infty} \frac{(sx)^k}{k}\right)$ which in turn is just \[\frac{1}{1-s}\left(\log\left(\frac{1}{1-x}\right) - \log\left(\frac{1}{1-sx}\right) \right)\] which gives the desired result. Now, to prove \eqref{critica} we proceed as follows. First notice that:
    \begin{align}
        m! W(\ell,n,m) = \sum_{\widetilde{\pi}} \sum_{\substack{(j_1,\ldots,j_m)\in \Z^m\\ j_1+\ldots+j_m = \ell \\ 0\leq j_i < |b_i|}} \prod_{i=1}^m (|b_i|-1)!
    \end{align}
    where the first sum runs over all the orderings $\widetilde{\pi}=(b_1,\ldots,b_m)$ of all elements $\pi=\{b_1,\ldots,b_m\}\in \mathscr{L}(n,m)$. This comes from the fact that for every such $\widetilde{\pi}$, if we choose how much weight to assign to each of the blocks, each block has its first element determined, and the remaining elements can be reordered in any fashion. Of course, this way we count every element of $\mathscr{W}(\ell,n,m)$ exactly $m!$ times. Taking out the product out of the second sum above, we get:
    \begin{align*}
        m! W(\ell,n,m) &= \sum_{\widetilde{\pi}} \left(\prod_{i=1}^m (|b_i|-1)! \right)\left| \left\{(j_1,\ldots,j_m)\in \Z^m : \sum_{i=1}^m j_i = \ell, 0\leq j_i < |b_i|\right\}\right|\\
        &= \sum_{\widetilde{\pi}} \left(\prod_{i=1}^m (|b_i|-1)! \right) [s^\ell]\left( \prod_{i=1}^{m}\sum_{j=0}^{|b_i|-1} s^{j} \right)\\
        &= [s^{\ell}] \sum_{\widetilde{\pi}} \left(\prod_{i=1}^m (|b_i|-1)! \sum_{j=0}^{|b_i|-1} s^{j}\right)
    \end{align*}
    Notice that the term inside the last sum does not take into account the whole element $\widetilde{\pi}=(b_1,\ldots,b_m)$, but only the size $|b_i|$ of each block. Thus, if we fix the sizes $|b_1|, \ldots, |b_m|$ of the blocks, we can recover exactly how many elements $\widetilde{\pi}$ have blocks of such sizes. Using multinomial coefficients, and abusing notation to write $b_i=|b_i|$:
    \begin{align*}
        m! W(\ell,n,m) &= [s^{\ell}] \sum_{\substack{(b_1,\ldots,b_m)\in\Z^m\\ b_1+\ldots+b_m = n \\ b_i \geq 0}} \binom{n}{b_1,\ldots,b_m}\left(\prod_{i=1}^m (b_i-1)! \sum_{j=0}^{b_i-1}
        s^{j}\right)\\
        &= [s^{\ell}] \sum_{\substack{(b_1,\ldots,b_m)\in\Z^m\\ b_1+\ldots+b_m = n \\ b_i \geq 0}} n! \left(\prod_{i=1}^m \frac{1}{b_i} \sum_{j=0}^{b_i-1} s_j\right)\\
        &= [s^\ell x^n] n! \left(\sum_{k=1}^n \frac{x^k}{k} (1+s+\ldots+s^{k-1})\right)^m,
    \end{align*}
    which proves \eqref{critica}.
\end{proof}

\begin{coro}\label{clave}
	For all $\ell,n,m$ one has:
	\[ W(\ell,n,m) = \sum_{j=0}^\ell \sum_{i=0}^{n-m} (-1)^{i+j} \binom{n}{j} \binom{m+\ell-j-1}{m-1} {j\brack {j-i}}{{n-j}\brack {m+i-j}}.\]
\end{coro}

\begin{proof}
	From the exponential generating function of the Stirling numbers of the first kind \cite[pg. 351]{knuth} one has:
		\[ {\alpha \brack \beta} = \frac{\alpha!}{\beta!} [x^\alpha] \left(\log\left(\tfrac{1}{1-x}\right)\right)^\beta.\]
	Now, using Theorem \ref{genfun}, we have the chain of equalities:
	\begin{align}
		W(\ell,n,m) &= \frac{n!}{m!}[x^n s^\ell] \left(\tfrac{1}{(1-s)^m} \left(\log\left(\tfrac{1}{1-x}\right) - \log\left(\tfrac{1}{1-sx}\right)\right)^m\right) \nonumber \\
		&= \frac{n!}{m!} [x^ns^\ell] \left(\frac{1}{(1-s)^m} \sum_{k=0}^m (-1)^k \binom{m}{k} \left(\log\left(\tfrac{1}{1-x}\right)\right)^{m-k} \left(\log\left(\tfrac{1}{1-sx}\right)\right)^k \right)\nonumber\\
		&= n! [x^n s^\ell] \left(\frac{1}{(1-s)^m} \sum_{k=0}^m (-1)^k \frac{\log\left(\tfrac{1}{1-x}\right)^{m-k}}{(m-k)!}\frac{\log\left(\tfrac{1}{1-sx}\right)^{k}}{k!} \right)\nonumber\\
		&= n! [s^\ell] \left(\frac{1}{(1-s)^m} \sum_{k=0}^m (-1)^k \sum_{j=0}^n [x^{n-j}]\left(\frac{\log\left(\tfrac{1}{1-x}\right)^{m-k}}{(m-k)!}\right) [x^j]\left( \frac{\log\left(\tfrac{1}{1-sx}\right)^{k}}{k!}\right) \right)\nonumber\\
		&= n! [s^\ell] \left(\frac{1}{(1-s)^m} \sum_{k=0}^m (-1)^k \sum_{j=k}^{n-m+k} [x^{n-j}]\left(\frac{\log\left(\tfrac{1}{1-x}\right)^{m-k}}{(m-k)!}\right) [x^j]\left( \frac{\log\left(\tfrac{1}{1-sx}\right)^{k}}{k!}\right) \right)\nonumber\\
		&= n![s^\ell] \left( \frac{1}{(1-s)^m} \sum_{k=0}^m (-1)^k \sum_{j=k}^{n-m+k} \frac{1}{(n-j)!} {{n-j}\brack{m-k}} \frac{1}{j!} s^j{j\brack k}\right)\nonumber\\
		&= [s^\ell]\left( \frac{1}{(1-s)^m} \sum_{k=0}^m\sum_{j=k}^{n-m+k} (-1)^k \binom{n}{j}s^j {{n-j}\brack{m-k}} {{j}\brack{k}}\right)\nonumber\\
		&= [s^\ell]\left( \frac{1}{(1-s)^m} \sum_{j=0}^n\sum_{k=j-m+n}^{j} (-1)^k \binom{n}{j}s^j {{n-j}\brack{m-k}} {{j}\brack{k}}\right)\nonumber\\
		&= [s^\ell]\left( \frac{1}{(1-s)^m} \sum_{j=0}^n\sum_{i=0}^{n-m} (-1)^{j-i} \binom{n}{j} {{n-j}\brack{m+i-j}} {{j}\brack{j-i}}s^j\right)\nonumber\\
		&= \sum_{j=0}^{\ell} \sum_{i=0}^{n-m}(-1)^{j-i} \binom{n}{j} \binom{m-1+\ell-j}{m-1} {{n-j}\brack{m+i-j}} {j\brack{j-i}}\nonumber
	\end{align}
    where in the fifth equation we changed the limits from $0\leq j \leq n$ to $k\leq j \leq n-m+k$ given that the coefficients of the first factor inside the sum are zero for degree $n-j < m-k$ and the coefficients of the second factor are zero for $j< k$.
\end{proof}

\section{The Proof of Theorem \ref{main}}

For $0\leq m\leq n-1$, we will call $e_{k,n,m}$ the coefficient of $t^m$ in the polynomial $E_{k,n}(t)$. Our aim is to show that all these $e_{k,n,m}$ are positive.\\

For $a,b,u$ integer numbers such that $u\geq 0$, we will denote $P_{a,b}^u$ the sum of all possible products of $u$ different integer numbers chosen in the interval of integers $[a,b]$. This is:
    \[ P_{a,b}^u := \sum_{a\leq x_1 < \ldots < x_u \leq b} x_1\cdot\ldots\cdot x_u.\]
It is easy to see that for $a=1$ one gets:
    \begin{equation}\label{prostir}
        P^u_{1,b} = {{b+1}\brack b+1-u},
    \end{equation}
where the brackets denote the (unsigned) Stirling numbers of the first kind \cite{stanley}.

\begin{lema}\label{coeficientes}
	The following formula holds:
		\[ e_{k,n,m} = \frac{1}{(n-1)!} \sum_{j=0}^{k-1} \sum_{i=0}^{n-m-1} (-1)^{i+j} \binom{n}{j} (k-j)^m {{n-j}\brack{m+1+i-j}}{j \brack {j-i}}.\]
\end{lema}

\begin{proof}
	We will work with the formula \eqref{formula}. Observe that:
	\begin{align*}
		[t^m] \binom{(k-j)t+n-1-j}{n-1} &= \frac{1}{(n-1)!} [t^m] \left( ((k-j)t + n-1-j)  \cdot \ldots \cdot ((k-j)t + 1 - j)\right)\\
		&= \frac{1}{(n-1)!} (k-j)^m P^{n-1-m}_{1-j,n-1-j},
	\end{align*}
    Observe that one has the following equality:
		\begin{align*} 
		P^{n-1-m}_{1-j,n-1-j} &= \sum_{i=0}^{n-m-1} P^i_{1-j,-1} P^{n-1-m-i}_{1,n-1-j}\\
		&= \sum_{i=0}^{n-m-1} (-1)^i P^i_{1,j-1} P^{n-1-m-i}_{1,n-1-j}.
		\end{align*}
	Therefore, using \eqref{prostir} we have that \[P_{1-j,n-1-j}^{n-m-1} = \sum_{i=0}^{n-m-1} (-1)^i {j\brack {j-i}} {{n-j}\brack {m+1+i-j}},\]
	so, in particular,
	\[ [t^m] \binom{(k-j)t+n-1-j}{n-1} =\frac{1}{(n-1)!}\sum_{i=0}^{n-m-1} (-1)^i (k-j)^m {j\brack {j-i}} {{n-j}\brack {m+1+i-j}}.\]
	The result follows easily from \eqref{formula} and this last identity.
\end{proof}

\begin{obs}
    If we use the shorter name \[f_{j,n,m}:=\sum_{i=0}^{n-m-1}(-1)^i {j\brack {j-i}} {{n-j}\brack {m+1+i-j}},\] we can rewrite the formula of Lemma \ref{coeficientes} as follows:
    \begin{equation} \label{coefi}
        e_{k,n,m} = \frac{1}{(n-1)!} \sum_{j=0}^{k} (-1)^j \binom{n}{j} (k-j)^m f_{j,n,m},
    \end{equation}
where we changed the upper limit of the sum since, when $j=k$, we are adding $0$.
\end{obs}

Now we are ready to state and prove the result from which our main theorem follows.

\begin{teo}
	For all $k,n,m$ as above, we have that:
		\[ e_{k,n,m} = \frac{1}{(n-1)!}\sum_{\ell=0}^{k-1} W(\ell,n,m+1) A(m,k-\ell-1),\]
	where $A(m,k-\ell-1)$ stands for the \textit{Eulerian numbers} \cite{knuth,stanley}. In particular $e_{k,n,m}$ is positive.
\end{teo}

\begin{proof}
	From equation \eqref{coefi} we can see that:
	\[ e_{k,n,m} = \frac{1}{(n-1)!} [x^k] F_{n,m}(x) \cdot G_m(x),\]
	where $F_{n,m}(x) := \sum_{j=0}^n (-1)^j \binom{n}{j} f_{j,n,m} x^j$ and $G_m(x) := \sum_{j=0}^{\infty} j^m x^j$. It is a well known consequence of the \textit{Worpitzky Identity} \cite{knuth} that:
		\[ G_m(x) = \frac{1}{(1-x)^{m+1}} \sum_{j=0}^m A(m,j) x^{j+1},\] 
	where $A(m,j)$ is an Eulerian number (the number of permutations of $m$ elements with exactly $j$ descents).\\
	
	So we have that the product $F_{n,m}(x)\cdot G_m(x)$ is equal to:
		\[ \frac{1}{(1-x)^{m+1}} F_{n,m}(x) \sum_{j=0}^m A(m,j) x^{j+1}.\]
	
	We compute the product of the first two factors. Let:
		\[ C_{n,m}(x) := \frac{1}{(1-x)^{m+1}} F_{n,m}(x),\]
	and observe that:
	\begin{align*}
		[x^{\ell}] C_{n,m}(x) &= [x^\ell] \left( \frac{1}{(1-x)^{m+1}} F_{n,m}(x)\right) \\
		&= \sum_{j=0}^{\ell} (-1)^j \binom{n}{j} f_{j,n,m} \binom{m+\ell-j}{m}\\
		&= \sum_{j=0}^\ell \sum_{i=0}^{n-m-1} (-1)^{i+j} \binom{n}{j}\binom{m+\ell-j}{m} {j\brack{j-i}} {{n-j}\brack{m+1+i-j}}\\
		&= W(\ell,n,m+1).
	\end{align*}
	where in the last step we used Corollary \ref{clave}. In particular $C_{n,m}(x)$ is a polynomial, and	the result now follows computing the product $C_{n,m}(x) \cdot \sum_{j=0}^m A(m,j) x^{j+1}$ to get the identity of the statement.
\end{proof}

\begin{obs}
    For small values of $n$ and $m$, the author at first observed that the series $C_{n,m}$ as defined in the above proof always happened to be a polynomial with positive coefficients, and that $C_{n,m}(1) = L(n,m+1)$. Then he tried to see how to define a statistic on the set $\mathscr{L}(n,m)$ in such a way that all these coefficients were captured, and it turned out that experimenting numerically the statistic defined as the weight of a linearly ordered partition did work for all small cases. The general proof was then carried out.
\end{obs}

\section{Acknowledgments}

The author wants to thank Davide Bolognini, Luca Moci and Matthias Beck for the useful suggestions and careful reading of the first draft, and Mat\'ias Hunicken for writing a part of the program that the author used to come up with the idea of defining the numbers $W(\ell,n,m)$. Also he wants to thank the reviewers of the article for the remarks and suggestions made to improve it. The author is supported by the Marie Sk{\l}odowska-Curie PhD fellowship as part of the program INdAM-DP-COFUND-2015.

\bibliography{bibliopaper} 
\bibliographystyle{plain}

\end{document}